\documentclass{amsart}
\usepackage{amsthm,amssymb,amsmath}
\usepackage{hyperref}

\newcommand{\F}{\mathbb{F}}

\newcommand{\OO}{\mathbb{O}}
\newcommand{\Q}{\mathbb{Q}}
\newcommand{\R}{\mathbb{R}}
\newcommand{\Z}{\mathbb{Z}}
\newcommand{\abs}[1]{\lvert#1\rvert}

\newcommand{\mc}[1]{\mathcal{#1}}

\newcommand{\msf}[1]{\mathsf{#1}}
\newcommand{\op}[1]{\operatorname{\msf{#1}}}
\newcommand{\vf}{\varphi}

\DeclareMathOperator{\tr}{\msf{tr}}
\swapnumbers
\newtheorem{theorem}{Theorem}
\newtheorem{lemma}[theorem]{Lemma}

\theoremstyle{definition}

\newtheorem{topic}[theorem]{}
\theoremstyle{remark}
\newtheorem{remark}[theorem]{Remark}
\begin{document}
%
%
%
%
%
%
\title[Octonions]{The octonions as a twisted group algebra}
\author{Tathagata Basak}
\address{Department of Mathematics\\Iowa State University, \\Ames, IA 50011}
\email{tathagat@iastate.edu}
\urladdr{http://orion.math.iastate.edu/tathagat}
\date{Jan 20, 2017}
\begin{abstract}
We show that the octonions can be defined as the $\R$-algebra
 with basis $\lbrace e^x \colon x \in \F_8 \rbrace$ and multiplication given by
 $e^x e^y = (-1)^{\vf(x,y)}e^{x + y}$, where $\vf(x,y) = \tr(y x^6)$.
While it is well known that the octonions can be described as a
twisted group algebra,
our purpose is to point out that this is a useful description.
We show how the basic properties of the octonions follow easily from our definition.
We give a uniform description of the sixteen orders of integral octonions containing
the Gravesian integers, and a computation-free proof of their existence.
\end{abstract}
\maketitle
Let $\F_q$ be a field with $q$ elements, and $\F_q^* = \F_q - \lbrace 0 \rbrace$.
Let $G$ be a group and let $R$ be a commutative ring. Fix a function $\sigma\colon G \times G \to R$.
Let
$R_{\sigma}[G]$ be the group algebra of $G$ twisted by $\sigma$. This means that
$R_{\sigma}[G]$ is a free $R$-module with a basis $\lbrace e^g \colon g \in G \rbrace$
indexed by the elements of $G$ with an $R$-bilinear multiplication map
$R_{\sigma}[G] \times R_{\sigma}[G] \to R_{\sigma}[G]$
defined by $e^x e^y = \sigma(x,y) e^{x + y}$.
\par
Let $\op{Fr}: \F_8 \to \F_8$ and $\tr: \F_8 \to \F_2$ be the maps
$\op{Fr}(x) = x^2$ and  $\tr(x) =  x + \op{Fr}(x) + \op{Fr}^2(x)$.
Define $\varphi: \F_8 \times \F_8 \to \F_2$ by
\begin{equation*}
\varphi(x,y) = \tr( y \bar{x} ) \text{\; where \;}
\bar{x} = x^6 = \begin{cases} x^{-1} & \text{ if \;} x \neq 0, \\
0 & \text{if \;} x = 0.
\end{cases}
\end{equation*}
Let $\sigma(x,y) = (-1)^{\varphi(x,y)}$.
We want to study the $R$-algebra $\OO(R) = R_{\sigma}(\F_8)$.
This algebra is not associative, since $\sigma$ is not a $2$-cocycle.
Clearly $e^0$ is the $2$-sided identity of $\OO(R)$.
Identify $R$ inside $\OO(R)$ by $x \mapsto x e^0$.
Write $\OO = \OO(\R)$.
We shall prove the following theorem.
\begin{theorem}
$\OO$ is the real non-associative division algebra of octonions.
\label{t-O}
\end{theorem}
While it is well known that $\OO$ can be described in the form $\R_{\sigma}[(\Z/2\Z)^3]$
for some $\sigma$ (see \cite{Ba:TO} or \cite{AM:QS}),
the formula for $\sigma$ given above seems to be new.
This definition makes certain symmetries visible: for example, we have
invariance under the Frobenius action and we have
$\sigma( a x, a y) = \sigma (x, y)$ for all $a \in \F_8^*$.
In the first part of this note,
we shall indicate how the basic properties of octonions become obvious
from our definition.
\par
In the second part of this note we study the orders of integral octonions containing $\OO(\Z)$.
It has long been known that there are sixteen such orders. Among them are
seven maximal ones (see \cite{Di:AN},\cite{Co:IC}), each forming a copy of the $E_8$ root lattice
with norm $\langle x , x \rangle = 2 x x^*$.
However, the definition of these maximal orders involve making an
awkward choice, and proving that they are multiplicatively closed
requires some unilluminating calculation (see \cite{C-Sm:OQ}).
A uniform description of the sixteen orders has been missing.
The description of $\OO$ as a twisted group algebra
allows us to give a uniform description of the $16$ orders
and a computation-free proof of their existence (see Theorem \ref{th-integral-octonions}).
\par
{\bf Acknowledgement:} I would like to thank Prof. Richard Borcherds and Prof. Daniel Allcock for encouraging
comments. I would specially like to thank Prof. Jonathan Smith for many interesting discussions and helpful suggestions.
\begin{topic}{\bf Dirac delta function:}
We begin with a little preliminary on the delta function $\delta_{x,y}$.
Let $\delta_{x} = \delta_{x,0}$.
Let $x_1, \dotsb, x_n$ be elements in an $\F_2$-vector space $V$.
Let $U = \sum_i \F_2 x_i$ and $\op{dim}_{\F_2}(U) = k$.
Let
$\op{ind} = \op{ind}_n : V^n \to \F_2$ be the function defined by
$\op{ind}_n(x_1 , \dotsb, x_n) = 1$
if and only if $k = n$.
For example,
$\op{ind}_1(x) = 1 + \delta_x$,
where the right hand side is considered as an element of $\F_2$.
We note the following useful formula for $\op{ind}_n(x_1, \dotsb, x_n)$:
\begin{equation}
\op{ind}_n(x_1, \dotsb, x_n)
= \sum_{\epsilon_1, \dotsb, \epsilon_n \in \F_2} \delta_{\epsilon_1 x_1 + \dotsb + \epsilon_n x_n}.
\label{eq-ind}
\end{equation}
To see why this is true, note that for each $u \in U$, the equation
$\sum_i \epsilon_i x_i    = u$ has $2^{n  - k}$ solutions for
$\epsilon_1, \dotsb, \epsilon_n$. So
the sum in \eqref{eq-ind} becomes $2^{n - k} \sum_{u \in U} \delta_u$.
\par
Let $\delta: \op{Fun}(V^k, \F_2) \to \op{Fun}( V^{k+1}, \F_2)$ be the coboundary map.
Note that
\begin{equation*}
(\delta \op{ind}_1)(x,y)  = 1 + \delta_x + \delta_y + \delta_{x +y}  =  \op{ind}_2(x,y) .
\end{equation*}
\end{topic}
Recall that we defined $\OO(R)$ as an $R$-algebra with basis $\lbrace e^x \colon x \in \F_8 \rbrace$
and multiplication defined by
\begin{equation}
e^x e^y = (-1)^{\vf(x,y)} e^{x  + y} \text{\; where \;} \vf(x,y) = \tr(y \bar{x} ) = \tr( y x^6).
\label{eq-def2}
\end{equation}
We record a few properties of the pairing $\vf$.
\begin{lemma} Let $x, y, z \in \F_8$. Then
\par
(a) $\vf(x, x) = x^7 =  1 + \delta_x = \op{ind}(x)$.
\par
(b)
$\vf(x, y) + \vf(y, x) = \tr(x \bar{y} + y \bar{x} ) = 1 + \delta_x + \delta_y + \delta_{x,y} = \op{ind}(x,y)$.
\par
(c) If $x, y , z \in \F_8^*$ such that $x  + y + z = 0$, then $\vf(x,y) = \vf(y,z) = \vf(z,x)$
\par
(d)  Assume  $z \in \F_8^*$. Then
$\vf(z x, z y) = \vf(x, y)$ and  $\vf(x^2 , y^2) = \vf( x, y)$.
In particular, if $\vf(x,y) = 0$, then $\vf(z x, z y) = 0$ and $\vf(x^2, y^2) = 0$.
\label{l-vf}
\end{lemma}
\begin{proof}
Let $v \in \F_8^*$. Then
$(v + 1) (1 + \tr(v + v^{-1} ) ) = (v+ 1) \sum_{t = 0}^6 v^t =  v^7 + 1 =0$.
It follows that $\tr(v + v^{-1} ) = 1$ since $v \neq 1$.
So if $x,y$ are distinct elements of $\F_8^*$, then
$\vf(x,y) + \vf(y,x)  = \tr(y x^{-1} + x y^{-1}) = 1$.
Part (b) now follows. Part (c) follows form part (b).
Part (d) is immediate from
the definition of $\varphi$ and the Frobenius invariance of trace.
\end{proof}
For $x,y, z \in \F_8^*$ with $x + y + z = 0$,
 Lemma~\ref{l-vf}(a),(b),(c) implies respectively:
\begin{equation}
(e^x)^2 = -1,  \text{\; \;} e^x e^y = -e^y e^x, \text{\; and \;} e^x e^y  = e^y e^z = e^z e^x.
\label{eq-ex}
\end{equation}
Now choose a generator $\alpha$ for the multiplicative group $\F_8^*$ satisfying
$\alpha^3 = \alpha + 1$.
So
$\tr(\alpha) = \alpha + \alpha^2 + \alpha^4 = 0$.
Write
\begin{equation*}
e_{\infty} = e^{0}, e_1 = e^{\alpha^1}, \; \dotsb \;, e_j = e^{\alpha^j}, \; \dotsb, e_7 = e^{\alpha^7}.
\end{equation*}
\par
When $j \neq \infty$, we think of the subscript $j$ modulo $7$, that is,
$e_7 = e_0$, $e_8 = e_1$ and so on\footnote{This notation is consistent with
\cite{C-Sm:OQ} where the basis elements are named $e_{\infty}, e_1, \dotsb, e_6, e_0$. But
there is an unfortunate clash of notation with \cite{Ba:TO} because Baez calls the identity $e_0$.}.
Note that $\vf(\alpha,\alpha^2) = \tr(\alpha^8) = \tr(\alpha) = 0$
and $\alpha + \alpha^2 = \alpha^4$, which translates into
$e_1 e_2 = e_4$.
%
%
%
Lemma~\ref{l-vf}(d) implies that if $e^x e^y = e^{x + y}$, that is $\vf(x,y) = 0$, then
$e^{\alpha x} e^{\alpha y} = e^{ \alpha (x + y)}$ and
$e^{x^2} e^{y^2} = e^{(x + y)^2}$. This translates into:
\begin{lemma}
Let $i, j , k \in \Z/7\Z$.
If $e_i e_j = e_k,$ then $e_{i + 1} e_{j + 1} = e_{k+1}$
and $e_{2 i} e_{2 j} = e_{2 k}$.
\label{l-indtr}
\end{lemma}
\begin{proof}[Proof of Theorem~\ref{t-O}]
From \cite[p.151]{Ba:TO} we know that
the relations in equation~\eqref{eq-ex} and Lemma~\ref{l-indtr}, together with the relations
$e_1 e_2 = e_4$ and $e_{\infty} e_j = e_j  = e_j e_{\infty}$ for $j \in \lbrace \infty, 1, \dotsb, 7 \rbrace$,
determine the multiplication table of $e_{\infty}, e_1, \dotsb, e_7$, and that this
agrees with the standard multiplication table of the octonions.
\end{proof}
Next, we derive the properties of the norm and the associator from our definition to
illustrate how the basic well known facts about
$\OO$ quickly follow from our definition.
One could go on in this vein and prove the Moufang law,
compute the derivations of $\OO$ and so on, but we shall omit these.
\begin{topic}{\bf Conjugate and norm: }
Let $a = \sum_{x \in \F_8} a_x e^x \in \OO$.
Write
$\op{Re}(a) = a_0 e^0$ and  $\op{Im}(a) = a - \op{Re}(a)$.
The {\it conjugate} of $a$ is $a^* = \op{Re}(a) - \op{Im}(a)$.
For $x \in \F_8$, we have
\begin{equation}
(e^x)^* = (-1)^{\delta_{x} + 1} e^x.
\label{eq-conj}
\end{equation}
Let $a, b \in \OO$.  We note three basic properties of conjugation:
\begin{equation}
(a b)^* = b^* a^*, \text{\;\;} a^{ ** } = a \text{\; and  \;} a^* = a \text{\; if and only if \;} a  = \op{Re}(a).
\label{eq-properties-of-conjugation}
\end{equation}
Because of $\R$-linearity, the identity $(a b)^* = b^* a^*$ needs to be
verified only for basis elements, and this is equivalent to Lemma~\ref{l-vf}(b).
The other two properties are clear from definitions.
The {\it trace} and {\it norm} of $a$ are defined as $\tr(a) = a + a^*$
and $N(a) = a a^*$ respectively.
Using \eqref{eq-properties-of-conjugation}, we have
$(a a^*)^* = a a^*$. So
$N(a) = a a^* = \op{Re}(a a^*) =  \sum_{x \in \F_8} a_x^2$.
So $a = 0$ if and only if $N(a) = 0$.
Now, let us verify that the norm is multiplicative. Let
$a = \sum_{x \in \F_8} a_x e^x$ and $b = \sum_{y \in \F_8} b_y e^y$
be two octonions.
From \eqref{eq-def2} and \eqref{eq-conj}, we compute
\begin{equation*}
N(ab) = \sum_{z,x,y} c_{x, y ,z}\,, \text{\; where \;} c_{x, y, z} =
a_x b_{z + x} a_y b_{z + y} (-1)^{\tr( z (\bar{x} + \bar{y})) + \delta_x + \delta_y}.
\end{equation*}
Using Lemma~\ref{l-vf},  we find
$c_{x,y,z + x + y}
=  (-1)^{ \tr(( x + y)(\bar{x} + \bar{y})) } c_{x,y,z}
= (-1)^{1 + \delta_{x,y} } c_{x,y,z}$.
So
\begin{equation*}
\sum_z c_{x, y , z} = \sum_{z} c_{x, y, z + x + y} = \sum_{z} - c_{x , y, z} \text{\; for \;} x \neq y.
\end{equation*}
So $\sum_{z} c_{x, y, z} = 0$ if $x \neq y$ and
$N(ab) = \sum_{x,z} c_{x,x,z} = \sum_{x, z} a_x^2 b_{z + x}^2 = N(a) N(b)$.
Now we have shown that $\OO$ is an $8$-dimensional ``real normed division algebra" with
$x^{-1} = x^*/N(x)$ for $x \in \OO - \lbrace 0 \rbrace$.
This gives another proof of Theorem~\ref{t-O}, since the octonions are known to be
the only such algebra (see \cite{Ba:TO}).
\end{topic}
\begin{topic}{\bf The associator:}
The {\it associator} $[a,b,c]$ of $a,b,c \in \OO$ is defined by
$[a,b,c] = (ab) c - a(bc)$.
Note that $e^x, e^y, e^z$ associate if and only the coboundary
\begin{equation*}
\delta \varphi(x,y,z) = \varphi(y,z) - \varphi(x + y, z) + \varphi(x, y + z) - \varphi(x,y)
= \varphi(y,z) - \varphi(x + y, z) + \varphi(x,z)
\end{equation*}
vanishes. Using Lemma~\ref{l-vf}(b) and equation \eqref{eq-ind} we find
\begin{equation}
\delta \varphi(x,y,z)  =  \op{ind}(x,y,z).
\label{eq-deltavarphi}
\end{equation}
So $(e^x e^y)e^z = \pm e^x (e^y e^z)$, where the plus sign holds if and only if
$\lbrace x,y,z \rbrace $ is a linearly dependent set, that is, $x,y,z$ all belong
to one of the 
copies of the quaternions in $\OO$.
Using \eqref{eq-def2} and \eqref{eq-deltavarphi},
and the fact that $\varphi(z,x) =\varphi(x,z) + 1$ if $x, y, z$
are independent, we get the following formula for the
associator:
\begin{equation}
[e^x,e^y,e^z] = (-1)^{\varphi(x,y)  + \varphi(y,z) + \varphi(z,x)} \bigl(1 -   (-1)^{\op{ind}(x,y,z)} \bigr) e^{x + y + z}.
\label{eq-associator}
\end{equation}
Equation \eqref{eq-associator} implies $[e^x,e^y,e^z]=-[e^x,e^z,e^y]$ and
 $[e^x,e^y,e^z]=[e^y,e^z,e^x]$ for all $x,y,z\in \F_8$.
 Using linearity, it follows that
$[a_{\sigma(1)} , a_{\sigma(2)} , a_{\sigma(3)} ] = \op{sign}(\sigma)[a_1, a_2, a_3]$
for all $a, b, c, \in \OO$ and for all permutations $\sigma \in S_3$.
In particular $[b,a,a] =  [a,a,b] = [a,b,a] = 0$ for all $a, b \in \OO$, so the octonions are {\it alternative}.
\end{topic}
\begin{remark}
There does not seem to be a definition of quaternion multiplication using the arithmetic of
$\F_4$ similar to \eqref{eq-def2}. If we take $\sigma: \F_4 \times \F_4 \to \F_2$ to be the function
$\sigma(x,y) = (-1)^{\tr(y x^2)}$, then $\Z_{\sigma}[\F_4]$
turns out to be the commutative associative algebra $\Z[x,y]/\langle x^2 - 1, y^2 - 1 \rangle$.
Of course the quaternions can be described as a twisted group algebra of $\F_2^2$,
see \cite{AM:QS}.
\end{remark}
\begin{topic}{\bf Orders of integral octonions: }
Let  $H$ denote the set of subsets of $\F_8$ having size $0$ or $4$ or $8$.
So $H  = \binom{\F_8}{4} \cup \lbrace \emptyset, \F_8 \rbrace$, where 
$\binom{\F_8}{j}$ is the set of subsets of $\F_8$ having size $j$.
\par
An octonion is called {\it integral} if it satisfies
a monic polynomial in $\Z[t]$.
Let $a = \sum_{x \in \F_8} a_x e^x \in \OO(\Q)$.
If $\op{Im}(a) \neq 0$, then the minimal polynomial of  $a$  is
$t^2 - \tr(a) t + N(a) \in \Q[t]$. So $a$ is integral if and only if
$\tr(a)$ and $N(a)$ belong to $\Z$.
Following \cite{C-Sm:OQ}, we define
the  {\it halving set} of $a$ 
to be the set of those $x \in \F_8$ for which $a_x \notin \Z$.
One observes immediately that if $a$ belongs to an order of integral octonions
containing $\OO(\Z)$,  then its halving set
lies in $H$, and $a_x$ is a half-integer for all $x$ in the halving set of $a$.
If $X \subseteq \F_8$, we define
\begin{equation*}
e^X = \sum_{x \in X} e^{x} \text{\; and \;} \sigma(X) = \sum_{x \in X} x.
\end{equation*}
\par
If $X \in H$, then  $e^X/2$ is an integral octonion with halving set $X$.
For example, suppose $\abs{X} = 4$. Write $a = e^X/2$.
  Then $N(a) = 1$, $\tr(a) = 1$
if $0 \in X$ and $\tr(a) = 0$ otherwise. 
So $a$ is an integral octonion
and
\begin{equation}
a^2 =
\begin{cases}
a - 1 & \text{\; if \;} 0 \in X \\
-1 & \text{\; otherwise}.
\end{cases}
\label{eq-unit-sq}
\end{equation}
For example, we could take $X$ to be a set of the form $\lbrace 0, x, y , z \rbrace$ where
$x + y + z = 0$. There are $7$ of these sets in $H$, and we call them {\it lines}
because they correspond to the lines of $P^2(\F_2)$.
The lines and line-complements are the size four subsets $X \subseteq \F_8$ such that
$\sigma(X) = 0$. Following \cite{C-Sm:OQ},
we shall call the remaining $56$ size four subsets {\it outer sets}.
The outer sets are a disjoint union of the following $7$ sets:
\begin{equation}
\mc{O}_x = \lbrace X \subseteq \F_8 \colon \abs{X} = 4, \sigma(X) = x \rbrace, \text{\; for \;} x \in \F_8^*.
\label{eq-def-Ox}
\end{equation}
\par
Let $E \subseteq \OO(\Q)$ be an order of integral octonions containing $\OO(\Z)$.
Let $a \in E$ with halving set $X$. Subtracting a suitable element of $\OO(\Z)$ from
$a$ we find that $e^X/2 \in E$. So $E$ contains the element
 $e^z ( e^X/2)$ which has halving set $z + X$. It follows that $E$ contains all
 octonions in $\tfrac{1}{2} \OO(\Z)$ whose halving sets are $\lbrace z + X \colon z \in \F_8 \rbrace$.
 So the halving sets of elements of $E$ must be a union of orbits of the translation action
 of the additive group of $\F_8$ on $H$.
The orbit space $\F_8 \backslash H$ has $16$ elements.
Two of the orbits are $\lbrace \emptyset \rbrace$
 and $\lbrace \F_8 \rbrace$. The lemma below describes the other fourteen orbits.
 \end{topic}
\begin{lemma}
(a) $X \in \binom{\F_8}{4}$ has nontrivial stabilizer in $\F_8$ if
and only if $X$ is a line or a line-complement.
\par
(b) $\binom{\F_8}{4}$ decomposes into seven orbits of size $2$ and seven orbits of size $8$.
\par
(c) For each line $L$, the set $\lbrace L, \F_8 - L \rbrace$ is an orbit of size $2$.
\par
(d) For each $x \in \F_8^*$, the set $\mc{O}_x$, defined in \eqref{eq-def-Ox} is
 an orbit of size $8$.
\label{l-translation}
\end{lemma}
\begin{proof}
Suppose $X \in \binom{\F_8}{4}$ such that $ 0 \in X$ and there exists $z \in \F_8^*$
with $z + X = X$. Write $X =\lbrace 0, x_1, x_2, x_3 \rbrace$. Then
$0 \in z + X =  \lbrace z, z + x_1 , z + x_2, z + x_3 \rbrace$.
So, without loss, we may assume $z + x_1 = 0$, that is, $z = x_1$.
Now $x_1 + x_2 = z + x_2 \in X$, so we must have $x_1 + x_2 = x_3$, that is,
$X$ is a line.
This proves parts (a) and (b), and shows in particular that $\F_8$ acts freely on
the outer sets. Part (c) is immediate.
To parametrize the size $8$ orbits in $\binom{\F_8}{4}$, let $X$ be an outer set
and $x = \sigma(X)$.
Then $x  \in \F_8^*$. 
Note that $\sigma(z + X) = x$ for all $z \in \F_8$. So $\mc{O}_x$ contains the $8$ elements of
$\lbrace z + X \colon z \in \F_8 \rbrace$. But the seven sets of the form $\mc{O}_x$ together have $56$ elements.
It follows that the sets $\mc{O}_x$ are precisely  the size $8$ orbits in $\binom{\F_8}{4}$.
\end{proof}
Let $\mc{O}$ be an orbit of the $\F_8$-action on $H$.
Let $\op{span}(\mc{O})$ be the smallest subset of $2^{\F_8}$
that contains $\mc{O}$ and is closed under symmetric difference.
We shall soon see that $\op{span}(\mc{O}) \subseteq H$.
 Now we can state the theorem
describing the 16 orders of integral octonions containing $\OO(\Z)$.
\begin{theorem}
Consider the translation action of the additive group of $\F_8$ on
 $H = \lbrace X \subseteq \F_8 \colon \abs{X} \in \lbrace 0, 4, 8 \rbrace \rbrace$.
For each orbit $\mc{O} \in \F_8 \backslash H$, there is an order $\mathbb{E}_{\mc{O}}$
of integral octonions containing $\OO(\Z)$ that can be described in any of the following three way:
\par
(a)  $\mathbb{E}_{\mc{O}}$ is the  $\Z$-span of $\lbrace  e^z \colon z \in \F_8 \rbrace \cup \lbrace e^X/2 \colon X \in \mc{O} \rbrace$.
\par
(b) $\mathbb{E}_{\mc{O}}$ consists of all octonions in $\tfrac{1}{2} \OO(\Z)$ whose halving sets are in $\op{span}(\mc{O})$.
\par
(c) Pick $X \in \mathcal{O}$. Then $\mathbb{E}_{\mc{O}}$ is the smallest $\Z$-subalgebra of $\OO(\Q)$ containing
$\OO(\Z)$ and $e^X/2$.
\par
Thus, the  $16$ orders of integral octonions containing $\OO(\Z)$ are
in one to one correspondence with the sixteen orbits in $\F_8 \backslash H$.
In particular, for each $ x \in \F_8^*$, we have an maximal order
$E_{\mc{O}_x}$.
\label{th-integral-octonions}
\end{theorem}
\begin{remark}
\cite{C-Sm:OQ} gives names for the $16$ octonion orders containing $\OO(\Z)$.
They correspond to  the sixteen orbits in $\F_8 \backslash H$ as follows:
The orbit $\lbrace \emptyset \rbrace$ corresponds to the Gravesian octaves $\OO(\Z)$.
The orbit $\lbrace \F_8 \rbrace$ corresponds to the Kleinian octaves.
The size $2$ orbits correspond to the double Hurwitzian rings.
 Finally, the orbits $\mathcal{O}_x$ of size $8$ correspond to the maximal orders
of octavian rings.
Let $j \in \lbrace 0, 1, \dotsb, 6\rbrace$.  In \cite{C-Sm:OQ}, 
the sets in  $\mathcal{O}_{\alpha^j}$ are called the
``outer $j$-sets" and the elements of the corresponding order are called ``$j$-integers".
\end{remark}
Before proving Theorem~\ref{th-integral-octonions}, we shall describe the sets $\op{span}(\mc{O})$.
First, define 
\begin{equation*}
\beta: 2^{\F_8} \times 2^{\F_8} \to \F_2 \text{\; by \;}
\beta(X,Y) = \abs{ X \cap Y} \bmod 2.
\end{equation*}
Let $\Delta$ denote the operation of symmetric difference of sets.
Identify the additive group $(2^{\F_8}, \Delta)$ with the additive group of the vector space $\F_2^8$.
Then $\beta$ is identified with the  standard symmetric (or alternating) bilinear form on $\F_2^8$.
If $M \subseteq 2^{\F_8}$, we shall write $M^{\bot}$ to be the orthogonal complement
of $M$ with respect to $\beta$.
\par
Let $\mc{O}$ be one of the sixteen orbits in $\F_8 \backslash H$ described above.
Then $\op{span}(\mc{O})$ is just the linear span of $\mc{O}$ in the vector space $2^{\F_8}$.
We immediately observe that if
$\mc{O} = \lbrace \emptyset \rbrace$, 
then $\op{span}(\mc{O}) = 
\{\emptyset\}
$,
if $\mc{O} = \lbrace \F_8 \rbrace$, then $\op{span}(\mc{O}) = \lbrace \emptyset, \F_8\rbrace$,
and
if $\mc{O}$ is an orbit of size $2$, then $\op{span}(\mc{O}) = \mc{O} \cup \lbrace \emptyset, \F_8\rbrace$. So $\op{span}(\mc{O})$ is a subspace of $2^{\F_8}$
of dimension $0$, $1$, $2$
respectively in these three cases.
The lemma below describes $\op{span}(\mc{O})$ for the orbits of size $8$.
\begin{lemma}
Let $z \in \F_8^*$ and let $\mc{O}_z$ be an orbit of size $8$ described in Lemma~\ref{l-translation}.
\par
(a) Let $X, Y \in \mc{O}_z$ such that $X$ is not equal to $Y$ or $\F_8 - Y$.
Then $\abs{X \cap Y} = 2$ and $X \Delta Y$ is a line containing $z$ or complement of such a line.
\par
(b) $\op{span}(\mc{O}_z)$ is a four dimensional subspace of $2^{\F_8}$
consisting of $\mc{O}_z \cup \lbrace \emptyset, \F_8\rbrace$
and the three lines containing $z$ and their complements.
\par
(c) $\op{span}(\mc{O}_z)$ is maximal isotropic in $2^{\F_8}$, that is,
$\op{span}(\mc{O}_z) = \op{span}(\mc{O}_z)^{\bot}$.
\par
(d) Let $X \subseteq \F_8$. Then $X \in \op{span}(\mc{O}_z)$ if and only if
$\abs{X \cap A} \equiv 0 \bmod 2$ for all $A \in \op{span}(\mc{O}_z)$.
\label{l-maximal-orbits}
\end{lemma}
\begin{proof}
Let $X, Y$ be as in part (a).
Translating by an element of $\F_8$ if necessary, assume first that $0 \in X \cap Y$.
Write $X = \lbrace 0, x_1, x_2, x_3 \rbrace$ and $Y = x + X$ for some $x \in \F_8^*$.
Note that $\sigma (X) = x_1 + x_2 + x_3 = z$.
Since $0 \in Y$, without loss, we may suppose that $ x + x_1 = 0$,
that is, $x = x_1$.
Then $Y = \lbrace x_1, 0, x_1 + x_2, x_1 + x_3 \rbrace$.
Now note that $(x_1 + x_2) \notin X$, since $X$ is not a line. For the same reason
$(x_1 + x_3) \notin X$.
It follows that $ X \cap  Y = \lbrace 0, x_1 \rbrace$ and that
$\F_8 - (X \Delta Y) = \lbrace 0, x_1, x_2 + x_3, z\rbrace$ is a line containing $z$. This proves part (a).
\par
Part (a) implies that if $X , Y \in \mc{O}_z$, then $\abs{X \cap Y}$ is even. In other words,
$\mc{O}_z \subseteq \mc{O}_z^{\bot} $.
It follows that $\op{span}(\mc{O}_z) \subseteq \op{span}(\mc{O}_z)^{\bot}$.
So $\op{span}(\mc{O}_z)$ is at most $4$ dimensional.
On the other hand, $\op{span}(\mc{O}_z)$ clearly has more than $8$ elements. It follow that
$\op{span}(\mc{O}_z)$ has dimension $4$, and hence it is
maximal isotropic.  This proves part (c). Part (d) is just a restatement of part (c).
\par
Since $\op{span}(\mc{O}_z)$ has $16$ elements, it suffices to show that
it contains the $16$ sets described in part (b).
From part (a) we know that $\op{span}(\mc{O}_z)$ contains
a line through $z$. Consider the action of $\op{GL}_3(\F_2)$ on $\F_8 \cong \F_2^3$.
The subgroup of $\op{GL}_3( \F_2)$ fixing $z$
preserves $\mc{O}_z$, and permutes the three lines through $z$. So
$\op{span}(\mc{O}_z)$ contains the three lines through $z$, and hence also their complements.
Part (b) follows.
\end{proof}
\begin{proof}[Proof of Theorem~\ref{th-integral-octonions}] Fix $x \in \F_8^*$. We shall prove the claims when
$\mc{O} = \mc{O}_x$ and $\mathbb{E}_{\mc{O}} = \mathbb{E}_{\mc{O}_x}$.
The other cases are similar and easier.
Fix $X \in \mathcal{O}_x$.
Let $\mathbb{E}$, $\mathbb{E}'$, $\mathbb{E}''$ respectively be the three
sets described in parts (a), (b) and (c) of the theorem.
By definition $\mathbb{E}''$ is a subring of $\OO(\Q)$, and
$\mathbb{E}'$ consists of integral octonions and contains $\OO(\Z)$. So it suffices to
show that $\mathbb{E}= \mathbb{E}' =  \mathbb{E}''$.
\par
Let $Y \in \mathcal{O}_x$. Then $ Y = y + X$ for
some $y \in \F_8$. Let $z \in \F_8$. We have
\begin{equation}
( e^Y /2) e^z= \tfrac{1}{2} \sum_{x \in X} \pm e^{x + y +  z} \equiv \tfrac{1}{2} e^{z + Y} \bmod \OO(\Z).
\label{eq-eYz}
\end{equation}
From \eqref{eq-eYz}, we find $\mathbb{E} . \OO(\Z) = \mathbb{E}$ and similarly $\OO(\Z) . \mathbb{E} =\mathbb{E}$.
Equation \eqref{eq-eYz} also implies $(e^{z + X}/2) = (e^{X} /2)e^z + c$ for some $c \in \OO(\Z)$.
So $\mathbb{E} \subseteq \mathbb{E}''$.
\par
Now let $Y, Z \in \mathcal{O}_x$. 
Then $Z = z + Y$ for some $z \in \F_8$. Write $a = e^Y/2$.
Then \eqref{eq-eYz} implies $e^{Z} /2=a e^z + b$ for some $b\in \OO(\Z)$.
From \eqref{eq-unit-sq}, note that $a^2 \in \mathbb{E}$.
Since $\OO$ is alternative, we get
\begin{equation*}
(e^Y/2 )(e^Z/2) = a^2 e^z + a b \in
\mathbb{E} .\OO(\Z) + \mathbb{E}. \OO(\Z) = \mathbb{E}.
\end{equation*}
So $\mathbb{E} . \mathbb{E} = \mathbb{E}$.
Since $\mathbb{E} \subseteq \mathbb{E}''$, it follows that $\mathbb{E} = \mathbb{E}''$.
The equality $\mathbb{E} = \mathbb{E}'$ follows
once we observe that if $u, v \in \mathbb{E}$, then the halving set of $u+v$ is the
symmetric difference of the halving sets of $u$ and $v$.
\end{proof}
\begin{remark}
Let $\mc{O}$ be one of the size $8$ orbits described in Lemma~\ref{l-translation}. Let $\mathbb{E}_{\mc{O}}$ be
the corresponding maximal order of integral octonions.
Equip the $\Z$-module $\mathbb{E}_{\mc{O}}$ with the bilinear form
$\langle x , y \rangle = \tr(x y^*)$. Then $\langle x, x \rangle = 2 N(x) \in 2 \Z$
for all $x \in \mathbb{E}_{\mc{O}}$.
So $\mathbb{E}_{\mc{O}}$ is an even integral lattice of rank $8$.
\par
Let $a, a' \in \tfrac{1}{2} \OO(\Z)$ with halving sets $X$ and $X'$ respectively.
Observe that $\tr(a a') \bmod \Z = \tfrac{1}{2} \abs{X \cap X'} \bmod \Z$.
In particular, if $a a'$ is an integral octonion, then we must have $\abs{X \cap X'} \equiv 0 \bmod 2$.
In view of this observation, self-duality of the code $\op{span}(\mc{O})$
(see Lemma~\ref{l-maximal-orbits}(d))
 immediately translates into self-duality of the even integral lattice $\mathbb{E}_{\mc{O}}$,
 and also implies that $\mathbb{E}_{\mc{O}}$  is a maximal order of integral octonions.
 In particular, we find that $\mathbb{E}_{\mc{O}}$ with the bilinear form 
$\langle x, y \rangle =  \tr(x y^*)$ 
 is the unique even integral self dual lattice $E_8$.
\end{remark}
\begin{remark}
Given a doubly even code $V$ and a ``factor set" $\varphi: V \times V \to \F_2$, Griess defined a Moufang loop structure on
$\F_2 \times V$ called a code loop (see \cite{Gr:CL}).
Code loops were defined in an effort to understand a Moufang loop discovered by Parker
and used by Conway to construct the monster simple group. When $V$ is the extended Golay code,
the code loop construction yields the Parker loop.
\par
In \eqref{eq-def2}, we have defined a Moufang loop structure on $\lbrace \pm e^x \colon x \in \F_8 \rbrace$. This loop is
 called the octonion loop.  Our definition of the octonion loop is very similar to the definition of a code loop.
The function $\op{ind}(x) = 1 + \delta_x$ plays the role of the 
function $q(x) = \tfrac{1}{4} \abs{x}$ used in the definition of a factor set.
Equation \eqref{eq-ind} gives us $(d \op{ind}) (x,y) = \op{ind}_2(x,y)$, 
$(d \op{ind})(x,y,z) = \op{ind}_3(x,y,z)$ and so on, where 
$d$ denotes the combinatorial polarization map 
introduced in \cite{Wa:CP}.
This makes Lemma \ref{l-vf}(a),(b) and equation \eqref{eq-deltavarphi} exactly analogous to the
the defining properties of a factor set (see \cite{Gr:CL}; definition 6 on page 225 and equation (*) on page 230).
\end{remark}
\end{document}